\documentclass[12pt]{article}

\usepackage{amsmath,amsthm,amssymb,graphicx,nicefrac}
\usepackage{cite}

\newcommand{\color}[1]{}

\newtheorem{thm}{Theorem}[section]
\newtheorem{lem}{Lemma}[section]

\newtheorem{cor}{Corollary}[section]

\newtheorem*{question}{Questions}

\newtheorem{Def}{Definition}[section]

\usepackage{amssymb}

\title{On the variety of linear recurrences and numerical semigroups}
\author{Ivan Martino\footnote{Stockholm University, Department of Mathematics, email: martino@math.su.se},
Luca Martino\footnote{Universidad Carlos III de Madrid, Department of Signal theory and Communications, email: luca@tsc.uc3m.es}}

\begin{document}
\maketitle
\begin{abstract}
In this work, we prove the existence of linear recurrences of order $M$ with a non-trivial solution vanishing exactly on the set of gaps (or a subset) of a numerical semigroup $S$ finitely generated by $a_1<a_2<\dots <a_N$ and $M=a_N$.

\end{abstract}

\section{Introduction and problem statement}

In this work, we study certain issues posed by R. Fr\"oberg and B. Shapiro in \cite{Ralf-Boris}.
Inspired by the Skolem-Mahler-Lech Theorem \cite{Lech}, they have defined the variety $V_{(M;I)}$, the set of all $M$-order linear recurrence equations with a non-trivial solution vanishing at least at all the points of a given non-empty finite set $I\subset \mathbb{N}$.
They have related the study of a particular open subvariety of $V_{(M;I)}$ to some ideals generated by Schur functions \cite{shur}.
They also stated certain open problems.
For instance, one open issue is to understand for which pairs $(M;I)$ the variety $V_{(M;I)}$ is empty or not. 

Here, we prove that the variety $V_{(M;I)}$ is non-empty when $I$ is a subset of the gaps of a numerical semigroup $S$ finitely generated by $a_1<a_2<\dots <a_N$ and $M=a_N$. We provide the analytic form of these suitable recurrence equations, jointly with the proper initial conditions. The solutions become zero only at the gaps of $S$. In the sequel, we recall briefly some useful background material and introduce more specifically our goal.

\subsection{Numerical semigroups}\label{sec-semigroup}
Numerical semigroups have been studied since the 19th century and they appear naturally in combinatorics and commutative algebra. In this work, we consider a numerical semigroup $S$ embedded in $(\mathbb{N}\cup\{0\},+)$.
%
%
%
%
Given $N$ integers  $a_1,\, a_2,\,\dots,\, a_N\in \mathbb{N}$, a (finitely generated) numerical semigroup $S$ \cite{Brauer}, is defined as 
\begin{equation*}
	S=\left\langle a_1,\, a_2,\,\dots,\, a_N\right\rangle=\left\{\sum_{i=1}^{N}n_ia_i: n_i\in \mathbb{N}\cup\{0\}\right\}.
\end{equation*}
If a natural number does not belong to $S$ it is called a \emph{gap} of $S$. We denote as $\Delta(a_1,\dots, a_N):=\mathbb{N} \setminus S$,
the set of the gaps of $S$. We have $\left|\Delta(a_1,\dots, a_N)\right| < \infty$ if and only if $\operatorname{gcd}(a_1, a_2, \dots, a_N)=1$ (in literature, it is often required as necessary condition). 
However, one can always reduce to this case. We \textbf{do not} require that $a_1, a_2,..., a_N$ are a minimal set of generators for $S$ and we always assume that $a_1< a_2<...< a_N$ and $\operatorname{gcd}(a_1,..., a_N)=1$. 

%
%
The gaps of a numerical semigroup $S$ are well studied \cite{Rosales} and strongly connected, for instance, with Frobenius number's problem 
and Hilbert function's problem \cite{Ramirez}. 
The maximal element of $\Delta(a_1,\dots, a_N)$ (with respect to the canonical order of $\mathbb{N}$) is called the {\em Frobenius number}.

%

%
%

\subsection{The variety of linear recurrences}\label{sec-recurrence}
We now present the open questions, stated in \cite{Ralf-Boris}, that we deal with in the sequel.
First of all, {\color{red}we associate to every $M$-tuple of complex numbers $\alpha=(\alpha_1,\dots, \alpha_M)$ the following linear recurrence equation $\mathcal{U}(\alpha)$:
\begin{equation}\label{dep_u_alpha}
  \mathcal{U}(\alpha): g_k+\alpha_1 g_{k-1}+\dots +\alpha_i g_{k-i}+\dots+ \alpha_Mg_{k-M}=0.
\end{equation}
}
If $\alpha_M\neq 0$, then $\mathcal{U}(\alpha)$ is of order $M$. 
%
%
The solutions of a linear homogeneous recurrence equation with constant coefficients are well-known \cite{Batchelder67} and {\color{red}they} depend on the roots of the \emph{characteristic polynomial} of $\mathcal{U}(\alpha)$,
\begin{equation*}
  p_{\alpha}(y): y^{M}+\alpha_1 y^{M-1}+\dots +\alpha_i y^i+\dots +\alpha_M=0.
\end{equation*}
Let $\{\rho_1,\dots, \rho_M\}$ be the set of the roots of $p_{\alpha}(y)$ (called \emph{characteristic roots} or \emph{poles}).~If~the roots $\rho_i$ are real and distinct, i.e., $\rho_i\in \mathbb{R}$, $\rho_i\neq \rho_j$ with $i\neq j$ and $i,j\in\{1,...,M\}$, a generic solution of the recurrence equation has the following analytic form 
\begin{equation} \label{IVAN_lasciaStoCazzoDiNumeroMiSERVE}
  g_k=c_1 \rho_1^k +c_2 \rho_2^k +\dots + c_M \rho_M^k,
\end{equation}
where the coefficients $c_i$ depend on the initial conditions associated to the recurrence equation \eqref{dep_u_alpha}. With multiple roots and complex roots, other functional forms appear in the solutions like cosine and sine functions \cite{Batchelder67}.

%

\begin{Def}
  Let $M\in \mathbb{N}$ and let $I$ be a {\color{red}non empty} finite subset of $\mathbb{N}$.
  The \emph{open linear recurrence variety} associated to the pair $(M;I)$, $V_{(M;I)}$, is the set of all linear recurrences of order exactly $M$ having a non-trivial solution vanishing at least in all the points of $I$. 
\end{Def}

Using the bijection $\mathcal{U}(-)$, we can always think the set of all linear recurrences (of order at most $M$) as the affine space $\mathbb{A}^{M}_{\mathbb{C}}$. 
Being of order exactly $M$ means that they belong to the affine principal open set $\mathbb{A}^{M}_{\mathbb{C}}\setminus \{\alpha_M\neq 0\}$.


In \cite{Ralf-Boris}, Fr\"oberg and Shapiro prove that $V_{(M;I)}$ is a algebraic variety and they ask the following questions:
\begin{question}
{\bf (a)} For which pairs $(M;I)$ the variety $V_{(M;I)}$ is empty/not empty? and {\bf (b)}, if $V_{(M;I)}\neq \emptyset$, is there a recurrence vanishing in a finite number of points?
\end{question}
In the rest of this work, we show that to each numerical semigroup $S =\left\langle a_1,\, a_2,\,\dots,\, a_N\right\rangle$ it is possible to associate a recurrence $\mathcal{U}_S$ of order $a_N$ vanishing \textit{exactly} on its {\color{red}\textit{finite number} of} gaps $\Delta(a_1,\dots, a_N)$, so that $V_{(a_N;\Delta(a_1,\dots, a_N))}\neq \emptyset$.
%

\section{The recurrence associated to a numerical semigroup}\label{sec-random-walk}
In this section, we first provide a novel characteristic function related to the semigroup $S$. 
Then, we construct the recurrence $\mathcal{U}_S$ associated to the semigroup $S =\left\langle a_1,\, a_2,\,\dots,\, a_N\right\rangle$.
Let $w_1, w_2, \dots, w_N$ be strictly positive real numbers.
Let us define the polynomial $F_1(z)=\sum_{i=1}^{N}w_i z^{a_i}$. 
%
We also set 
\begin{equation}\label{TransfFunctEq}
  G(z)=\frac{1}{1-F_1(z)}= \frac{1}{1-w_1z^{a_1}-...-w_Nz^{a_N}}.
\end{equation}
We denote by $g_k$ the coefficient of $z^k$ in the power series expansion of $G(z)$, i.e., $G(z)=\sum_{k=0}^{+\infty} g_k z^k$.

\begin{lem}\label{lem-utile}
    The sequence of coefficients, $\{g_k\}_{k\in \mathbb{N}\cup\{0\}}$, is {\color{blue} the} solution of the recurrence
    \begin{equation}\label{EqRecurrence}
      \mathcal{U}_S: \,\,\, g_k=w_1g_{k-a_1} + ... +w_Ng_{k-a_N}, \quad  \forall \mbox{  } k> 0,
    \end{equation}
     with initial condition $g_0=1$ and  
     {\color{magenta}  $g_{j}=0$, for $-a_N<j<0$}.
\end{lem}
\begin{proof}
  We can rewrite Eq. (\ref{TransfFunctEq}) as $G(z)(1-w_1z^{a_1}-...-w_Nz^{a_N})=1$. Since $G(z)=\sum_{k=0}^{+\infty} g_k z^k$, replacing above, we obtain easily $\sum_{k=0}^{+\infty} g_{k} z^{k}-w_1\sum_{k=0}^{+\infty} g_{k} z^{k+a_1}...-w_N\sum_{k=0}^{+\infty} g_{k} z^{k+a_N} =1$, and also
  \begin{eqnarray}\label{eq-utile}
    \sum_{k=0}^{+\infty} g_{k} z^{k}-w_1\sum_{i=a_1}^{+\infty} g_{i-a_1} z^{i}...-w_N\sum_{j=a_N}^{+\infty} g_{j-a_N} z^{j} =1. \nonumber 
\end{eqnarray}
Now, setting  $g_{j}=0$, for $-a_N<j<0$, we can also rewrite the left-side of the previous equation as 
$g_0+\sum_{k=1}^{+\infty} \left(g_{k}-w_1g_{k-a_1} ... -w_Ng_{k-a_N} \right) z^{k} =1$.
Finally, note that to hold the equality we need that $g_0=1$ and $g_{k}-w_1g_{k-a_1} ... -w_Ng_{k-a_N}=0$.
\end{proof}
%
%

The recurrence in (\ref{EqRecurrence}) is denoted by $\mathcal{U}_S$ and it is \emph{associated} to the semigroup $S =\left\langle a_1,\, a_2,\,\dots,\, a_N\right\rangle$. 
In the sequel, we link this result with the questions stated in the previous section. We recall that we do not require that $a_1,\, a_2,\,\dots,\, a_N$ is a minimal set of generators of $S$.

\begin{lem}\label{lem-gk-zero}
  The coefficient $g_k$ is zero if and only if $k\not\in S$. 
\end{lem}
\begin{proof}
%
  Using that $\nicefrac{1}{1-x}=\sum_{i\geq 0} x^i$, then  $G(z)= \sum_{k=0}^{\infty} g_k z^k=\sum_{t=0}^{\infty} F_t(z)$
  where we have set $F_t(z)=\left[F_1(z)\right]^{t}$ and $F_0(z)=1$. 
  The coefficients of $F_t(z)=\left(\sum_{i=1}^{N}w_i z^{a_i}\right)^t$ are non-zero only on the $z$-powers having for exponent an element of the semigroup given by a sum of $t$ generators of $S$ (non necessarily different), that is $\sum_{q=1}^{t}a_{i_q}$ where $1\leq i_q\leq N$.
  Indeed
  \begin{eqnarray}
    F_t(z)	&=&\left(\sum_{i=1}^{N}w_i z^{a_i}\right)^t=\sum\prod_{q=1}^{t}w_{i_q} z^{a_{i_q}}=\sum \left(\prod_{q=1}^{t}w_{i_q}\right) z^{\left(\sum_{q=1}^{t}a_{i_q}\right)}.
    \label{eq-Ft-expression}
  \end{eqnarray}
  For this reason, in the sum $\sum_{t=0}^{\infty} F_t(z)$ the exponents of the power $z^i$ with non zero coefficients are exactly all the elements of $S$.
  Therefore, one gets the statement.
\end{proof}

\begin{thm}\label{thm-main-theorem}
  Let $S =\left\langle a_1,\, a_2,\,\dots,\, a_N\right\rangle$ and let $I\subseteq \Delta(a_1, a_2, \dots, a_N)$.
  Then $V_{(\beta; I)}\neq \emptyset$, for all $\beta\in S$, with $\beta\geq a_N$.
%
\end{thm}
\begin{proof}
  For every choice of strictly positive real numbers $\{w_i\}_{i=1}^{N}$, the recurrence equation $\mathcal{U}_S$, given in (\ref{EqRecurrence}), belongs to $V_{(a_N;I)}$.
  Indeed, comparing Eqs. \eqref{dep_u_alpha} and \eqref{EqRecurrence}, we note easily that     $\alpha_j=-w_k$ if $j = a_k$ and zero otherwise, 
  with $w_{N}\neq 0$, so that $\alpha_{a_N} \neq 0$.  
  Hence $\mathcal{U}_S$ is a recurrence equation of order $a_N$.
  %
  Using Lemma \ref{lem-gk-zero}, we know that a solution $\{g_k\}_{k\in\mathbb{N}}$ is zero if and only if $k\not\in S$. 
  This proves the result for $\beta=a_N$.  
  For $\beta>a_N$, we observe that the semigroup $S$ does not change if we add to the generators the element $\beta\in S$. 
  Since we have never required that $\{a_1,\dots, a_N\}$ is a minimal set of generators for $S$, we apply again the previous theorem with $a_1<a_2<\dots <a_N<\beta$.
\end{proof}

%

We have seen that for any  finitely generated numerical semigroup $S$ such that $\operatorname{gcd}(a_1, a_2, \dots, a_N)=1$, the Frobenius number, $g(S)$, exists and every integer $k$ greater than $g(S)$ belongs to $S$.
Then, we could narrow down the previous result:

\begin{cor}\label{cor-non-empty-definitely}
  Let $S =\left\langle a_1,\, a_2,\,\dots,\, a_N\right\rangle$ and let $I\subseteq \Delta(a_1, a_2, \dots, a_N)$.
  Then there exists a constant value $K\in \mathbb{N}$ such that for all $\beta>K$, we have $V_{(\beta; I)}\neq \emptyset$.
\end{cor}

We can also easily provide certain informations about the dimension of the variety  {\color{magenta}$V_{(a_N,I)}$}.
We remark that in algebraic geometry one often uses the so-called {\em Krull dimension} instead of the topological dimension. A suitable definition is given in \cite{AM-comm}, for instance. However, in this article, the reader can suppose that the Krull dimension is the topological one, because we work on the complex field.

\begin{cor}
  Let $S =\left\langle a_1,\, a_2,\,\dots,\, a_N\right\rangle$ and let $I\subseteq \Delta(a_1, a_2, \dots, a_N)$.
  Then the Krull dimension of $V_{(a_N;I)}$, 
  is at least $N$, i.e., $\operatorname{dim}_{\mathbb{C}}(V_{(a_N;I)})\geq N$.
\end{cor}
\begin{proof}
  We remark that $V_{(a_N;I)}$ is an open {\em complex} algebraic variety (since defined by polynomial equation and by $\alpha_{a_N} \neq 0$, see \cite{Ralf-Boris}). 
  Let $W$ be the subset of $V_{(a_N;I)}$ defined by the recurrences (\ref{EqRecurrence}) for every choice of strictly positive real number $\{w_1,\dots, w_N\}$.
  In Theorem \ref{thm-main-theorem}, we have proved that $V_{(a_N;I)}\neq \emptyset$ by showing that $W\neq \emptyset$.
  
  We observe that $W$ is isomorphic to $\mathbb{R}_{>0}^N$. 
  Each complex algebraic variety that contains the non-algebraic subset $\mathbb{R}_{>0}^N$ has a complex sub-variety $C$ containing $W$ and of Krull dimension at least $N$, $W\subset C\subseteq V_{(a_N;I)}$.
  Thus $\operatorname{dim}_{\mathbb{C}}(V_{(a_N;I)})\geq N$.
\end{proof}

\subsection{Probabilistic interpretation}\label{sec-probabilistic}
If the coefficients $w_i\geq 0$, are chosen such that $\sum_{i=1}^N w_i=1$,
they define a probability mass, and the functions  $F_t(z)=\left[F_1(z)\right]^{t}$ with $F_1(z)=\sum_{i=1}^{N}w_i z^{a_i}$, defined in the proof of Lemma \ref{lem-gk-zero}, and $G(z)$ have a probabilistic interpretation. 
Let $X_t$ be a discrete random variable taking values in $\mathbb{N}$, and $t\in \mathbb{N}$. 
We can define, for instance, a random walk associated to the semigroup $S=\left\langle a_1,\, a_2,\,\dots,\, a_N\right\rangle$ as  $ X_t=X_{t-1}+a_i$, with probability $w_i$,  $i=1,...,N$, starting with $X_0=0$. 
The function $F_t(z)$ represents the {\it probability generating function} (PGF) \cite{DeGroot02} associated to the probability of visiting the state $k$ exactly at the time instant $t$, 
$f_{t,k}=\mbox{Prob}\{X_t=k\}$, i.e., $F_t(z)=\sum_{k=0}^{t\cdot a_N}f_{t,k}z^k$.
Let us also consider now the probability of  visiting the state $k$, i.e.,
\[
  g_k=\operatorname{Prob}\{X_t=k \mbox{  } \mbox{ for some } \mbox{  } t\in \mathbb{N}\}, \quad k\in \mathbb{N}.
\]
Basic statistical considerations \cite{DeGroot02} lead us to write the PGF $G(z)$ corresponding to these probability as  $G(z)=\sum_{t=1}^{\infty} F_t(z)$ and so we obtain
$G(z)=\frac{1}{1-w_1z^{a_1}-...-w_Nz^{a_N}}$.
The probabilities $\{g_k\}_{k\in \mathbb{N}}$ of visiting a state $k$, satisfy the linear recurrence equation $\mathcal{U}_S$. 
%
They are zero if and only if these states $k$ coincides exactly with the gaps of the numerical semigroup $S$, i.e., $k\in\Delta(a_1, a_2, \dots, a_N)$.
%

%
%
%

\section*{Acknowledgements}
We would like to thank Ralf Fr\"oberg (Stockholm University, Department of Mathematics) and Boris Shapiro (Stockholm University, Department of Mathematics) for their help and several suggestions.
Furthermore, we thank the reviewers for their useful comments which have helped us to improve this manuscript. Moreover, this work has been partially supported by Government of Spain (project COMONSENS, id.CSD2008-00010, project DEIPRO, and project COMPREHENSION, id. TEC2012-38883-C02-01).

\addcontentsline{toc}{section}{Bibliografy}

\end{document}